\documentclass[oneside,english]{amsart}
\usepackage[T1]{fontenc}
\usepackage[latin9]{inputenc}
\usepackage{amsthm}
\usepackage{amssymb}
\usepackage{stmaryrd}

\makeatletter
\numberwithin{equation}{section}
\numberwithin{figure}{section}
\theoremstyle{plain}
\newtheorem{thm}{\protect\theoremname}
\theoremstyle{plain}
\newtheorem{lem}[thm]{\protect\lemmaname}
\theoremstyle{plain}
\newtheorem{cor}[thm]{\protect\corollaryname}

\makeatother

\usepackage{babel}
\providecommand{\corollaryname}{Corollary}
\providecommand{\lemmaname}{Lemma}
\providecommand{\theoremname}{Theorem}

\begin{document}

\title{On Uniform Approximations of Normal Distributions By Jacobi Theta
Functions}

\author{Ruiming Zhang}

\address{College of Science\\
Northwest A\&F University\\
Yangling, Shaanxi 712100\\
P. R. China.}

\email{ruimingzhang@yahoo.com}

\keywords{theta function; normal distributions.}

\thanks{The work is supported by the National Natural Science Foundation
of China grants No. 11371294 and No. 11771355. The author also thanks
Professor Mourad E. H. Ismail for mentioning this problem to him in
a private communication.}

\subjclass[2000]{33D05; 33C45.}
\begin{abstract}
In this short note we study uniform approximations to the normal distributions
by Jacobi theta functions. We shall show that scaled theta functions
approach to a normal distribution exponentially fast. 
\end{abstract}

\maketitle

\section{Introduction}

In probability theory, under mild conditions such as the existence
of second moment of the distribution, the central limit theorem (CLT)
establishes when the normalized sum of independent random variables
tends toward a normal distribution even if the original variables
themselves are not normally distributed. The theorem is a key concept
in probability theory because it implies that probabilistic and statistical
methods that work for normal distributions can be applicable to many
problems involving other types of distributions, \cite{Durret,Grinstead}.
Therefore, given an nonnegative integrable special function on the
real line, it is interesting to know under what scaling it approximates
a normal distribution, and how fast the approximation is. The Jacobi
theta functions are fundamental special functions in both mathematics
and physics, for example, please see Wikipedia(https://en.wikipedia.org/wiki/Theta\_function)
and DLMF(https://dlmf.nist.gov/20). In probability theory many important
probability density functions are also weight functions for q-orthogonal
polynomials, \cite{Szab=000142owski1,Szab=000142owski2}, and some
of them are expressible in terms of Jacobi theta functions. For example,
in \cite{Szab=000142owski2} the q-Gaussian density is expressed as
a product of $\theta_{3}$ and $\theta_{2}$. In this short note we
shall prove that scaled Jacobi theta functions tend to a normal distribution
exponentially fast. 

For all $m\in\mathbb{N}$ and $a,a_{1},\dots a_{m},q\in\mathbb{C}$
with $|q|<1$, let \cite{Andrews,Berndt} 
\begin{equation}
(a;q)_{\infty}=\prod_{k=0}^{\infty}\left(1-aq^{k}\right),\quad(a_{1},\dots,a_{m};q)_{\infty}=\prod_{k=1}^{m}(a_{k};q)_{\infty}.\label{eq:1.1}
\end{equation}
 Let $q=e^{\pi i\tau},\ \Im(\tau)>0$, the four Jacobi theta functions
are defined by \cite{Andrews,Rademacher}
\begin{align}
\theta_{1}(v\vert\tau)= & 2q^{\frac{1}{4}}\sin\pi v\left(q^{2},zq^{2},q^{2}/z;q^{2}\right)_{\infty}=-i\sum_{n=-\infty}^{\infty}(-1)^{n}q^{(n+1/2)^{2}}e^{(2n+1)\pi iv},\label{eq:1.5}\\
\theta_{2}(v\vert\tau)= & 2q^{\frac{1}{4}}\cos\pi v\left(q^{2},-zq^{2},-q^{2}/z;q^{2}\right)_{\infty}=\sum_{n=-\infty}^{\infty}q^{(n+1/2)^{2}}e^{(2n+1)\pi iv},\label{eq:1.6}\\
\theta_{3}(v\vert\tau)= & \left(q^{2},-zq,-q/z;q^{2}\right)_{\infty}=\sum_{n=-\infty}^{\infty}q^{n^{2}}e^{2\pi inv},\label{eq:1.7}\\
\theta_{4}(v\vert\tau)= & \left(q^{2},zq,q/z;q^{2}\right)_{\infty}=\sum_{n=-\infty}^{\infty}(-1)^{n}q^{n^{2}}e^{2\pi inv}.\label{eq:1.8}
\end{align}
 Then,
\begin{align}
\theta_{1}(v\vert\tau) & =i\sqrt{\frac{i}{\tau}}\exp\left(-\frac{\pi iv^{2}}{\tau}\right)\theta_{1}\left(\frac{v}{\tau}\big|-\frac{1}{\tau}\right),\label{eq:1.9}\\
\theta_{2}(v\vert\tau) & =i\sqrt{\frac{i}{\tau}}\exp\left(-\frac{\pi iv^{2}}{\tau}\right)\theta_{4}\left(\frac{v}{\tau}\big|-\frac{1}{\tau}\right),\label{eq:1.10}\\
\theta_{3}(v\vert\tau) & =\sqrt{\frac{i}{\tau}}\exp\left(-\frac{\pi iv^{2}}{\tau}\right)\theta_{3}\left(\frac{v}{\tau}\big|-\frac{1}{\tau}\right)\label{eq:1.11}\\
\theta_{4}(v\vert\tau) & =\sqrt{\frac{i}{\tau}}\exp\left(-\frac{\pi iv^{2}}{\tau}\right)\theta_{2}\left(\frac{v}{\tau}\big|-\frac{1}{\tau}\right).\label{eq:1.12}
\end{align}

\section{Main Results and Proofs}

For each $x\in\mathbb{R}$, we let $[x]$ be integer part of $x$
and $\left\{ x\right\} $ be the fractional part of $x$, i.e.
\begin{equation}
x=[x]+\{x\},\quad[x]\in\mathbb{Z},\ 0\le\{x\}<1.\label{eq:2.1}
\end{equation}
Furthermore, for each $x\in\mathbb{R}$ we let
\begin{equation}
\left(\left(x\right)\right)=\{x\}-\frac{1}{2},\label{eq:2.2}
\end{equation}
then it is clear that 
\begin{equation}
-\frac{1}{2}\le\left(\left(x\right)\right)<\frac{1}{2}.\label{eq:2.3}
\end{equation}
 Additionally, for each $x\in\mathbb{R}$ we let 
\begin{equation}
x=m_{x}+\left\llbracket x\right\rrbracket ,\quad m_{x}\in\mathbb{Z},\ -\frac{1}{2}\le\left|\left\llbracket x\right\rrbracket \right|<\frac{1}{2}.
\end{equation}

\begin{lem}
\label{lem:2.1}For each fixed $v\in\mathbb{R}$ and $\Im(\tau)>0$
we have
\begin{align}
\theta_{1}(v\vert\tau) & =\left(-1\right)^{\left[v\right]}\sqrt{\frac{i}{\tau}}\exp\left(-\frac{\pi i((v))^{2}}{\tau}\right)\theta_{4}\left(\frac{((v))}{\tau}\big|-\frac{1}{\tau}\right)\label{eq:2.4}\\
\theta_{2}(v\vert\tau) & =i\left(-1\right)^{m_{v}}\sqrt{\frac{i}{\tau}}\exp\left(-\frac{\pi i\left\llbracket v\right\rrbracket ^{2}}{\tau}\right)\theta_{4}\left(\frac{\left\llbracket v\right\rrbracket }{\tau}\big|-\frac{1}{\tau}\right)\label{eq:2.5}\\
\theta_{3}(v\vert\tau) & =\sqrt{\frac{i}{\tau}}\exp\left(-\frac{\pi i\left\llbracket v\right\rrbracket ^{2}}{\tau}\right)\theta_{3}\left(\frac{\left\llbracket v\right\rrbracket }{\tau}\big|-\frac{1}{\tau}\right)\label{eq:2.6}\\
\theta_{4}(v\vert\tau) & =\sqrt{\frac{i}{\tau}}\exp\left(-\frac{\pi i((v))^{2}}{\tau}\right)\theta_{3}\left(\frac{((v))}{\tau}\big|-\frac{1}{\tau}\right).\label{eq:2.7}
\end{align}
 
\end{lem}

\begin{proof}
By (\ref{eq:1.5}) and (\ref{eq:1.9}) and shift the summation index
on the second line by $\left[v\right]$ we have

\begin{align*}
 & \theta_{1}(v\vert\tau)=\sqrt{\frac{i}{\tau}}\sum_{n=-\infty}^{\infty}\left(-1\right)^{n}\exp\left(-\frac{\pi i}{\tau}\left(\left(n+\frac{1}{2}\right)^{2}-2\left(n+\frac{1}{2}\right)v+v^{2}\right)\right)\\
 & =\sqrt{\frac{i}{\tau}}\sum_{n=-\infty}^{\infty}\left(-1\right)^{n}\exp\left(-\frac{\pi i}{\tau}\left(n+\frac{1}{2}-\left[v\right]-\left\{ v\right\} \right)^{2}\right)\\
 & =\left(-1\right)^{[v]}\sqrt{\frac{i}{\tau}}\sum_{n=-\infty}^{\infty}\left(-1\right)^{n}\exp\left(-\frac{\pi i}{\tau}\left(n-((v))\right)^{2}\right)\\
 & =\left(-1\right)^{[v]}\sqrt{\frac{i}{\tau}}\exp\left(-\frac{\pi i((v))^{2}}{\tau}\right)\sum_{n=-\infty}^{\infty}\left(-1\right)^{n}\exp\left(-\frac{n^{2}\pi i}{\tau}+2n\pi i\frac{((v))}{\tau}\right)\\
 & =\left(-1\right)^{[v]}\sqrt{\frac{i}{\tau}}\exp\left(-\frac{\pi i((v))^{2}}{\tau}\right)\theta_{4}\left(\frac{((v))}{\tau}\big|-\frac{1}{\tau}\right),
\end{align*}
which gives (\ref{eq:2.4}). (\ref{eq:2.5}) can be proved by applying
(\ref{eq:1.6}) and (\ref{eq:1.10}),

\begin{align*}
 & \theta_{2}(v\vert\tau)=i\sqrt{\frac{i}{\tau}}\sum_{n=-\infty}^{\infty}\left(-1\right)^{n}\exp\left(-\frac{n^{2}\pi i}{\tau}+\frac{2nv\pi i}{\tau}-\frac{\pi iv^{2}}{\tau}\right)\\
 & =i\sqrt{\frac{i}{\tau}}\sum_{n=-\infty}^{\infty}\left(-1\right)^{n}\exp\left(-\frac{\pi i}{\tau}\left(n-v\right)^{2}\right)\\
 & =i\sqrt{\frac{i}{\tau}}\sum_{n=-\infty}^{\infty}\left(-1\right)^{n}\exp\left(-\frac{\pi i}{\tau}\left(n-m_{v}-\left\llbracket v\right\rrbracket \right)^{2}\right)\\
 & =i\left(-1\right)^{m_{v}}\sqrt{\frac{i}{\tau}}\sum_{n=-\infty}^{\infty}\left(-1\right)^{n}\exp\left(-\frac{\pi i}{\tau}\left(n-\left\llbracket v\right\rrbracket \right)^{2}\right)\\
 & =i\left(-1\right)^{m_{v}}\sqrt{\frac{i}{\tau}}\exp\left(-\frac{\pi i}{\tau}\left\llbracket v\right\rrbracket ^{2}\right)\sum_{n=-\infty}^{\infty}\left(-1\right)^{n}\exp\left(-\frac{n^{2}\pi i}{\tau}+\frac{2n\pi i\left\llbracket v\right\rrbracket }{\tau}\right)\\
 & =i\left(-1\right)^{m_{v}}\sqrt{\frac{i}{\tau}}\exp\left(-\frac{\pi i}{\tau}\left\llbracket v\right\rrbracket ^{2}\right)\theta_{4}\left(\frac{\left\llbracket v\right\rrbracket }{\tau}\big|-\frac{1}{\tau}\right).
\end{align*}
Let $v\in\mathbb{R}$ and $v=m_{v}+\left\llbracket v\right\rrbracket $
, by(\ref{eq:1.7}) and (\ref{eq:1.11}) we get
\begin{align*}
 & \theta_{3}(v\vert\tau)=\sqrt{\frac{i}{\tau}}\sum_{n=-\infty}^{\infty}\exp\left(-\frac{\pi i}{\tau}\left(n^{2}-2nv+v^{2}\right)\right)\\
 & =\sqrt{\frac{i}{\tau}}\sum_{n=-\infty}^{\infty}\exp\left(-\frac{\pi i}{\tau}\left(n-v\right)^{2}\right)=\sqrt{\frac{i}{\tau}}\sum_{n=-\infty}^{\infty}\exp\left(-\frac{\pi i}{\tau}\left(n-\left\llbracket v\right\rrbracket \right)^{2}\right)\\
 & =\sqrt{\frac{i}{\tau}}\exp\left(-\frac{\pi i}{\tau}\left\llbracket v\right\rrbracket ^{2}\right)\sum_{n=-\infty}^{\infty}\exp\left(-\frac{\pi in^{2}}{\tau}+\frac{2\pi in\left\llbracket v\right\rrbracket }{\tau}\right),
\end{align*}
which gives (\ref{eq:2.6}). By (\ref{eq:1.6}) and (\ref{eq:1.12})
we get 
\begin{align*}
 & \theta_{4}(v\vert\tau)=\sqrt{\frac{i}{\tau}}\sum_{n=-\infty}^{\infty}\exp\left(-\frac{\pi i}{\tau}\left(n+\frac{1}{2}-v\right)^{2}\right)=\sqrt{\frac{i}{\tau}}\sum_{n=-\infty}^{\infty}\exp\left(-\frac{\pi i}{\tau}\left(n-((v))\right)^{2}\right)\\
 & =\sqrt{\frac{i}{\tau}}\exp\left(-\frac{\pi i}{\tau}((v))^{2}\right)\sum_{n=-\infty}^{\infty}\exp\left(-\frac{\pi in^{2}}{\tau}+\frac{2\pi in((v))}{\tau}\right),
\end{align*}
which is (\ref{eq:2.7}).
\end{proof}
\begin{thm}
\label{thm:2.2} For any positive number $a$, let $a>t>0$. Then
we have
\begin{align}
\left(-1\right)^{\left[v\right]}t^{\frac{1}{2}}e^{\pi((v))^{2}/t}\theta_{1}(v\vert it) & =1-2e^{-\pi/t}\cosh\frac{2\pi((v))}{t}+R_{11}(v,t),\label{eq:2.8}\\
\left(-1\right)^{m_{v}}t^{\frac{1}{2}}e^{\pi\left\llbracket v\right\rrbracket ^{2}/t}\theta_{2}(v\vert it) & =1-2e^{-\pi/t}\cosh\frac{2\pi\left\llbracket v\right\rrbracket }{t}+R_{12}(v,t),\label{eq:2.9}\\
t^{\frac{1}{2}}e^{\pi\left\llbracket v\right\rrbracket ^{2}/t}\theta_{3}(v\vert it) & =1+2e^{-\pi/t}\cosh\frac{2\pi\left\llbracket v\right\rrbracket }{t}+R_{13}(v,t),\label{eq:2.10}\\
t^{\frac{1}{2}}e^{\pi((v))^{2}/t}\theta_{4}(v\vert it) & =1+2e^{-\pi/t}\cosh\frac{2\pi((v))}{t}+R_{14}(v,t),\label{eq:2.11}
\end{align}
where 
\begin{equation}
\left|R_{1j}(v,t)\right|\le\frac{2e^{-2\pi/t}}{1-e^{-\pi/a}},\quad j=1,2,3,4.\label{eq:2.12}
\end{equation}
\end{thm}

\begin{proof}
Take $\tau=it,\ t>0$ in (\ref{eq:2.6}) to obtain 
\begin{align*}
 & t^{\frac{1}{2}}e^{\pi\left\llbracket v\right\rrbracket ^{2}/t}\theta_{3}(v\vert it)=\sum_{n=-\infty}^{\infty}\exp\left(-\frac{n^{2}\pi}{t}+\frac{2\pi n\left\llbracket v\right\rrbracket }{t}\right)\\
 & =1+e^{-\pi/t+2\pi\left\llbracket v\right\rrbracket /t}+e^{-\pi/t+2\pi\left\llbracket v\right\rrbracket /t}+R_{11}(v,t)\\
 & =1+2e^{-\pi/t}\cosh\frac{2\pi\left\llbracket v\right\rrbracket }{t}+R_{11}(v,t),
\end{align*}
where
\[
R_{11}(v,t)=\sum_{n=2}^{\infty}\left\{ \exp\left(-\frac{\pi}{t}\left(n^{2}+2n\left\llbracket v\right\rrbracket \right)\right)+\exp\left(-\frac{\pi}{t}\left(n^{2}-2n\left\llbracket v\right\rrbracket \right)\right)\right\} .
\]
Then by (\ref{eq:2.3}) and $n^{2}\ge2n$ for $n\ge2$ we have
\begin{align*}
 & \left|R_{11}(v,t)\right|\le2\sum_{n=2}^{\infty}\exp\left(-\frac{\pi}{t}\left(n^{2}-n\right)\right)\le2\sum_{n=2}^{\infty}\exp\left(-\frac{n\pi}{t}\right)\\
 & \le2e^{-2\pi/t}\sum_{n=2}^{\infty}\exp\left(-\frac{n\pi}{a}\right)=\frac{2e^{-2\pi/t}}{1-e^{-\pi/a}},
\end{align*}
which proves (\ref{eq:2.10}). (\ref{eq:2.8}),(\ref{eq:2.9}) and
(\ref{eq:2.11}) can be proved similarly. 
\end{proof}
\begin{cor}
For any positive number $C$, let $x\in\mathbb{R}$ and $|x|\le C$.
Then for any positive number $\epsilon$ with $0<\epsilon<\min\left\{ 1,2C\right\} $
and $0<t<\frac{\epsilon^{2}}{4\pi^{2}C^{2}}<1$ we have 
\begin{align}
t^{\frac{1}{2}}\theta_{1}\left(\frac{1}{2}+\sqrt{t}x\vert it\right) & =e^{-\pi x^{2}}\left(1+R_{21}(x,t)\right),\label{eq:2.13}\\
t^{\frac{1}{2}}\theta_{2}\left(\sqrt{t}x\vert it\right) & =e^{-\pi x^{2}}\left(1+R_{22}(x,t)\right),\label{eq:2.14}\\
t^{\frac{1}{2}}\theta_{3}\left(\sqrt{t}x\vert it\right) & =e^{-\pi x^{2}}\left(1+R_{23}(x,t)\right),\label{eq:2.15}\\
t^{\frac{1}{2}}\theta_{4}\left(\frac{1}{2}+\sqrt{t}x\vert it\right) & =e^{-\pi x^{2}}\left(1+R_{24}(x,t)\right),\label{eq:2.16}
\end{align}
 where 
\begin{equation}
\left|R_{2j}(x,t)\right|\le\frac{4-2e^{-\pi}}{1-e^{-\pi}}e^{-(\pi-\epsilon)/t},\quad j=1,2,3,4.\label{eq:2.17}
\end{equation}
\end{cor}

\begin{proof}
First we observe that 
\[
\left|x\sqrt{t}\right|<\frac{\epsilon}{2}<\frac{1}{2},\quad1>\frac{1}{2}\pm x\sqrt{t}>0,\quad0<t<1.
\]
Then by (\ref{eq:2.2}) we have
\[
\left[\frac{1}{2}\pm x\sqrt{t}\right]=0,\quad m_{\pm x\sqrt{t}}=0,\quad\left\llbracket \pm x\sqrt{t}\right\rrbracket =\pm x\sqrt{t}
\]
and
\[
\left(\left(\frac{1}{2}\pm x\sqrt{t}\right)\right)=\left\{ \frac{1}{2}\pm x\sqrt{t}\right\} -\frac{1}{2}=\pm x\sqrt{t}.
\]
Then,
\[
t^{\frac{1}{2}}e^{\pi x^{2}}\theta_{1}\left(\frac{1}{2}+x\sqrt{t}\vert it\right)=1+R_{21}(x,t)
\]
and
\[
t^{\frac{1}{2}}e^{\pi x^{2}}\theta_{4}\left(\frac{1}{2}+x\sqrt{t}\vert it\right)=1+R_{24}(x,t),
\]
where
\[
R_{21}(x,t)=-2e^{-\pi/t}\cosh\frac{2\pi x}{\sqrt{t}}+R_{11}\left(\frac{1}{2}+x\sqrt{t},t\right)
\]
 and
\[
R_{24}(x,t)=2e^{-\pi/t}\cosh\frac{2\pi x}{\sqrt{t}}+R_{14}\left(\frac{1}{2}+x\sqrt{t},t\right).
\]
 Since
\[
\cosh\frac{2\pi x}{\sqrt{t}}\le e^{2\pi|x|/\sqrt{t}}\le e^{2\pi C/\sqrt{t}}
\]
and
\[
\left|R_{2j}(x,t)\right|\le2e^{-\pi/t+2\pi C/\sqrt{t}}+\frac{2e^{-2\pi/t}}{1-e^{-\pi}},\quad j=1,4.
\]
Since for $t<\frac{\epsilon^{2}}{4\pi^{2}C^{2}}$ we have $2\pi C/\sqrt{t}\le\frac{\epsilon}{t}$,
thus for $j=1,4$ we have
\[
\left|R_{2j}(x,t)\right|\le2e^{-\pi/t+\epsilon/t}+\frac{2e^{-2\pi/t}}{1-e^{-\pi}}\le\frac{4-2e^{-\pi}}{1-e^{-\pi}}e^{-(\pi-\epsilon)/t}.
\]
 which proves (\ref{eq:2.13}) and (\ref{eq:2.16}). (\ref{eq:2.14})
and (\ref{eq:2.15}) can be proved similarly.
\end{proof}

\end{document}